\documentclass[a4paper,11pt,reqno,noindent]{amsart}
\usepackage[centertags]{amsmath}
\usepackage{amsfonts,amssymb,amsthm} 
\usepackage{hyperref}

 \hypersetup{
     pdfmenubar=false,        % show Acrobat¡¯s menu?
     pdfnewwindow=true,      % links in new window
     colorlinks=false,       % false: boxed links; true: colored links
     linkcolor=blue,          % color of internal links
     citecolor=blue,        % color of links to bibliography
     filecolor=magenta,      % color of file links
     urlcolor=cyan           % color of external links
 }
\usepackage{graphicx} 

\usepackage[english]{babel}
\usepackage{color}
\usepackage[body={15cm,21.5cm},centering]{geometry} 
\usepackage{fancyhdr}
\usepackage{esint}
\usepackage{enumerate}
\usepackage[numbers,sort]{natbib}
\newtheorem{theorem}{Theorem}
\newtheorem{lemma}{Lemma}
\newtheorem{proposition}{Proposition}
\newtheorem{corollary}{Corollary}
\newtheorem{remark}{Remark}
\newtheorem{definition}{Definition}

\theoremstyle{definition}

\newtheorem{example}{Example}

\usepackage[active]{srcltx}

\newcommand{\R}{{\mathbb R}}

\newcommand{\son}{\mathrm{SO}(n)}

\newcommand{\Pcal}{{\mathcal P}}

\newcommand{\N}{{\mathbb N}}

\newcommand{\non}{{\nonumber}}

\newcommand{\A}{{\mathcal A}}

\newcommand{\e}{\varepsilon}
\newcommand{\per}{{\mathrm {Per }}}

\renewcommand{\d}{\,{\mathrm d}}
\newcommand{\dx} {\,{\mathrm d}x}

\renewcommand{\H}{{\mathcal H}}
\renewcommand{\S}{{\mathcal S}}
\newcommand{\B}{{\mathcal B}_h}
\newcommand{\G}{{\mathcal G}_h}
\newcommand{\hausd}{{\mathcal {H}} }
\newcommand{\leb}{{\mathcal {L}} }

\newcommand{\id}{\mathrm{id}}
\newcommand{\dist}{{\mathrm dist}\,}
\newcommand{\Ik}{I^{\mathrm Kh.}\,}
\newcommand{\Iko}{I_{3d}^{\mathrm Kh.}\,}

\newcommand{\sff}{{\mathrm II}}

\renewcommand{\deg}{{\mathrm{ deg}}\,}
\renewcommand{\det}{{\mathrm{det}}\,}
\newcommand{\kap}{\,{\mathrm{cap}}}
\newcommand{\sgn}{\,{\mathrm{sgn}}}
\newcommand{\supp}{\,{\mathrm{ supp}}}
\newcommand{\diam}{{\mathrm{diam}}\,}

\newcommand\wto{\rightharpoonup}

\newcommand{\insieme}[1]{\left \{#1\right \}}

\newcommand{\Lb}{{\mathcal L}}

\begin{document}
%%-----------------------------
%%      the top matter
%%-----------------------------
\title{Interpenetration of matter in plate theories obtained as
  $\Gamma$-limits}\thanks{This work has been completed while the second author
  held a Hausdorff scholarship (PhD grant) at the Hausdorff Center for Mathematics in Bonn.}% At most 5 thanks
\author{Heiner Olbermann}\address{Hausdorff Center for Mathematics \& Institute for Applied Mathematics,
University of Bonn,
Endenicher Allee 60,
53115 Bonn, Germany, {\tt heiner.olbermann@hcm.uni-bonn.de}}
\author{Eris Runa}\address{Max Planck Institute for Mathematics in the Sciences,
  Inselstrasse 22, 04103 Leipzig, Germany, {\tt eris.runa@mis.mpg.de}}
%\author{...}\address{...}
%
\date{\today}
\begin{abstract}
We reconsider the derivation of plate theories as $\Gamma$-limits of $3$-dimensional nonlinear elasticity and define a suitable notion for the interpenetration of matter in the limit configuration. This is done via the Brouwer degree.
For the approximating maps, we adopt as definition of interpenetration of matter
the notion of non-invertibility almost everywhere, 
see J.  M.  Ball: Global invertibility of Sobolev functions and the interpenetration of matter, Proc.\ Roy.\ Soc.\ Edinburgh Sect.\ A, 88(3-4):315--328, 1981. Given a limit map satisfying the former interpenetration property, we show that any recovery sequence (in the sense of $\Gamma$-convergence) has to consist of maps that satisfy the latter interpenetration property except for finitely many sequence elements. Then we explain how our result is applied in the context of the derivation of plate theories. 
\end{abstract}
\maketitle
%%-----------------------------
%%      your text
%%-----------------------------
\section*{Introduction}
In the mathematical theory of nonlinear elasticity, the elastic deformations of
an elastic body
are identified with (almost-) minimizers of some free elastic
energy % \footnote{In the sequel, we will call the free elastic energy  ``energy'' for short.}
functional. This identification works as follows: The reference configuration of the elastic body is  some domain
$\Omega\subset\R^n$,  the deformation is a map $y:\Omega\to\R^m$, and the
associated  energy $I:X\to\R$ has as domain the function space of
deformations $y$. % In this setting, not only minimizers of $I$ are of physical
% interest. There exist many relevant models where minimizers do not even exist,
% and  there are various other reasons why one can take interest in
% configurations  whose free energy does not attain the minimum. 
Of crucial importance is the right choice for the
 function space $X$. Unphysical
deformations (e.g., non-injective maps, which represent configurations
displaying self-penetration of matter)
should either be excluded from $X$, or  the  energy of these configurations
should be infinite, signaling that it is not possible to observe them in the
``real world''. There exists a large amount of literature on how to choose the
function space of elastic deformations in a manner  that at the same time
excludes unphysical configurations and ensures existence of energy minimizers. We do not attempt to give an exhaustive
literature review here, and only 
 mention
\cite{MR0475169,MR0478899,MR703623,MR1346364,MR862546,MR1393418,MR616782}.
In \cite{MR1346364}, a framework has been
introduced that allows for \emph{cavitation}, i.e., the free energy allows for
the formation of holes in the elastic body. Cavity formation can be observed in
experiments;  the mathematical theory for radially symmetric cavities has
been developed in \cite{MR703623}. In \cite{MR1346364}, the function space $X$
is chosen such that cavities created at one point cannot be filled with matter
from elsewhere. Clearly, this is another property that ``physical'' deformations
of an elastic body should fulfill. The mathematical formulation of this condition (called ``(INV)''
in  \cite{MR1346364}) is rather technical. \\
\\
An important question in nonlinear elasticity is the relation between models in
three, two and one dimensions. Conceptually and mathematically, the most
satisfying approach is the derivation of lower dimensional models from a
3-dimensional one by $\Gamma$-convergence \cite{MR1201152}. In
\cite{MR1916989,MR2210909}, a hierarchy of 2-dimensional plate models has been
derived from 3-dimensional nonlinear elasticity. 
These models can be classified by the assumed scaling of the energy per unit
thickness $I_h$ in the underlying 3d theory, where $h$ denotes the thickness of the elastic sheet. Assuming $I_h\sim h^{\beta}$, where $h$ is the thickness of the
elastic plate, the $\Gamma$-limit for $\beta=2$ is nonlinear bending theory
\cite{MR1916989}. The parameter choice $2<\beta<4$ results in
``von-K\'arm\'an-like'' plate theories, see \cite{MR2210909}.\\
\\
The 3d models taken as a
starting point for this hierarchy of $\Gamma$-limits do not require
the condition (INV). 
In \cite{MR1346364} it is shown that in general, if
condition (INV) is not imposed,
 it is possible to construct sequences of (almost everywhere) invertible deformations of finite energy that weakly converge
to a  non-(a.~e.) invertible one. We will give a slightly more  detailed presentation of this construction in Section~\ref{sec:Iae}. 
What matters for us is that such a situation is  potentially problematic for the derivation of plate theories by $\Gamma$-convergence:  A weakly converging sequence  of invertible (a.~e.) functions might result in a non-invertible (a.~e.)
configuration with finite elastic energy in the 2d limit theory. 
% In the latter,
% deformations are $W^{2,2}$-isometric immersions, but not necessarily
% embeddings.
The obvious
cure would be, of course, to impose condition (INV) on the 3d theory.
In the present contribution, we show that this is not necessary.\\ 
\\
In contrast to
 the existing mathematical literature on interpenetration of matter that
mainly focuses on finding sufficient conditions for invertibility of elastic
deformations, we here identify sufficient conditions for non-invertibility. Questions related to the image of Sobolev functions are known
to be a delicate issue, and these objects may display counter-intuitive features,
cf.~the pathological examples going back to Besicovitch
\cite{MR0036825,MR1310951}. Here, such pathologies are not problematic, because we want to show that the image of the considered functions is sufficiently large.\\
\\
This will be achieved in the main theorem of the present paper, Theorem~\ref{mainthm}. We will assume the typical conditions fulfilled by sequences of elastic
deformations of thin films in the derivation of 3d-to-2d
$\Gamma$-limits. Additionally, we will assume that the limit configuration is
non-invertible in a suitable sense, see Definition~\ref{def:simple_interpenetration}.  This
definition is crucial for our method of proof to be workable. % In a number of
% remarks after that definition, we  explain  why it is  an appropriate assumption
% for our purposes.
The statement of Theorem~\ref{mainthm} is that under these
assumptions, the considered sequence $y_h$ of elastic deformations must consist of
non-invertible functions for $h$ small enough as $h\to 0$. 

The structure of the present paper is as follows. 
In section~\ref{sec:results}, we state our main result. 
In section~\ref{setting}, we recall some results from the literature that we will use for its proof.  
The proof of the theorem (see section~\ref{sec:pf_mainthm}) is
based on a reduction to a  2-dimensional domain. The intersection on a
sufficiently large set in the 2d-domain is proved by a homotopy argument, and
the passage back to the 3-dimensional situation is performed with the help of
the geometric rigidity result by Friesecke, James and M\"uller
\cite{MR1916989}. In Section~\ref{sec:appl}, we recall the derivation of plate
theories as  $\Gamma$-limits of 3d-nonlinear elasticity, and obtain some
straightforward corollaries from the application of Theorem~\ref{mainthm} to
these settings. \\
\\
{\bf Notation.} The symbol $C$ will be used  as follows. A statement such as ``$f\leq Cg$'', where
$f,g$ are quantities that depends on a variable $x$, is to be read as
: There exists a numerical constant $C>0$ with the property that $f\leq C
g$ for all $x$. The value of $C$ may change from one line to the next.\\
The $d$-dimensional Lebesgue measure is denoted by $\Lb^d$.
Further we write $\omega(m)=\Gamma(1/2)^m/\Gamma(m/2+1)$; if
$m\in\N$, then $\omega(m)$ is the volume of the
$m$-dimensional ball.
% The application to nonlinear bending  is an immediate consequence of the main theorem, see Corollary~\ref{thm:nlbendinginter}. In Section~\ref{sec:vKcom}, we will briefly explain how the proof  has to be modified to obtain an analogous result for von-K\'arm\'an-like plate theories. 
%% Secondly, we would like to point out the connection to the work on non-Euclidean
%% plate theories carried out in \cite{MR2859870}. It was shown in the latter paper
%% that for non-Euclidean reference metrics (metrics with non-zero curvature), non-trivial $\Gamma$-limits that can
%% be identified with plate theories only exist if the energy scaling with the
%% plate thickness $h$ is not too strong. More precisely, the $\Gamma$-limit of
%% $h^{-\beta} I^h$ for $h\to 0$ is trivial in this setting if $\beta>2$, where $I^h$ is the
%% elastic energy per unit thickness.  Here, we look at reference metrics with
%% zero curvature, but the associated stress-free reference configuration is one of
%% infinite energy (since it is self-intersecting). We show that in this case, one
%% still has the hierarchy of $\Gamma$-limits for $2<\beta<4$ as in \cite{MR2210909}.
\section{Statement of results}
\label{sec:results}
\subsection{Brouwer degree}
\label{sec:Brouwer}
First we need to recall the definition and some basic properties of the Brouwer degree.
For $ U\subset\R^n$ bounded, $f\in C^\infty(\bar U,\R^n)$, and $y\in \R^n\setminus f(\partial U)$ such that $\det\nabla f(x)\neq 0$ for all $x\in f^{-1}(y)$, the Brouwer degree is defined by
\begin{equation*} %---{{{
   \begin{split}\label{eq:det_formulation_degree}
      \deg(f, U,y)=\sum_{x\in f^{-1}(y)}\sgn(\det\nabla f(x))\,.
   \end{split}
\end{equation*} %---}}}
One can show that for $\varphi\in C^\infty_0(\R^n)$ with $\supp(\varphi)\cap f(\partial U)=\emptyset$, and any $y\in\R^n$ in the same connected component of $\R^n\setminus f(\partial U)$ as $\supp(\varphi)$, 
\[
\deg(f,  U,y)\int_{\R^n}\varphi(z)\d z=\int_U\varphi(f(x))\det\nabla f(x)\d x\,.
\]
By this formula and approximation by smooth functions, one can define the degree
for any continuous  $f\in C^0(\bar  U,\R^n)$ and $y\not\in f(\partial U)$. One
can show that the degree only depends on $f|_{\partial U}$. Hence, from now on,
we write $\deg(f,\partial U, y)\equiv \deg(f, U, y)$. Another basic property of the degree is
\[
\deg(f, \partial U,y)\neq 0\quad\Rightarrow \quad y\in f( U).
\]
On each connected component of $\R^n\setminus f(\partial U)$,
$\deg(f, \partial U,\cdot)$ is constant. The latter yields the implication
\begin{equation}
\label{eq:9}
y_0\in \partial\{y\in\R^n:\deg(f, \partial U,y)=k\}\quad\Rightarrow\quad y_0\in
f(\partial U)\,,
\end{equation}
for any $k\in\N$. Finally, we will need the homotopy invariance of the degree: If $\gamma:[0,1]\to\R^n$  and $H:[0,1]\times
\overline U\to\R^n$ are continuous, and $\gamma(t)\not\in H(t,\partial U)$ for
$t\in [0,1]$, then
\begin{equation}
  \label{eq:111}
  \deg(H(0,\cdot),\partial U,\gamma(0))=  \deg(H(1,\cdot),\partial U,\gamma(1))\,.
\end{equation}
For the details of the definition and the proofs of the properties mentioned here, we refer to \cite{MR787404}.

\subsection{Invertibility almost everywhere and the example by M\"uller and Spector}
\label{sec:Iae}
Next we introduce appropriate notions
of invertibility for Sobolev functions. 
\begin{definition}[Invertibility almost everywhere \cite{MR616782,MR1346364}]
\label{def:invert}
Let $U\subset\R^n$, and let $f$ be (a representative of an equivalence class) in
$W^{1,1}(U,\R^n)$. 
%% \achtung{(M\"uller, Spector
%% : $W^{1,1}$. But $L^1$ is enough, correct?)} 
We say $f$ is 
invertible almost everywhere if there is a null
set $N\subset U$ such that $f|_{U\setminus N}$ is injective.
\end{definition}
\noindent
Note that invertibility almost everywhere only depends on the equivalence class.\\
\\
In \cite{MR1346364}, M\"uller and Spector gave an example of a sequence of
a.~e.~invertible maps that weakly converge to a map that is $2$-to-$1$ on a set of
positive measure. (As in~\cite{MR1346364}, by saying that a map $u$ is $2$-to-$1$ at a point $x$ we mean that there exists exactly one point $\bar{x}\neq x$ such that $u(x)=u(\bar{x})$.) 
Their examples were two-dimensional, but similar (slightly
more complicated) constructions can be carried out in higher dimensions
too. Crucial for their construction is the assumed regularity. The formation of
cavities must be permitted, which is the case if the deformations are $W^{1,p}$
with $p<n$, where $n$ is the dimension of the domain. We do not give the
explicit formulas for the examples, but only give a qualitative explanation and
refer to Figure~\ref{fig:prova}, where the construction is sketched. 

\begin{figure}
   
   \begin{center}
      \includegraphics[scale=.8]{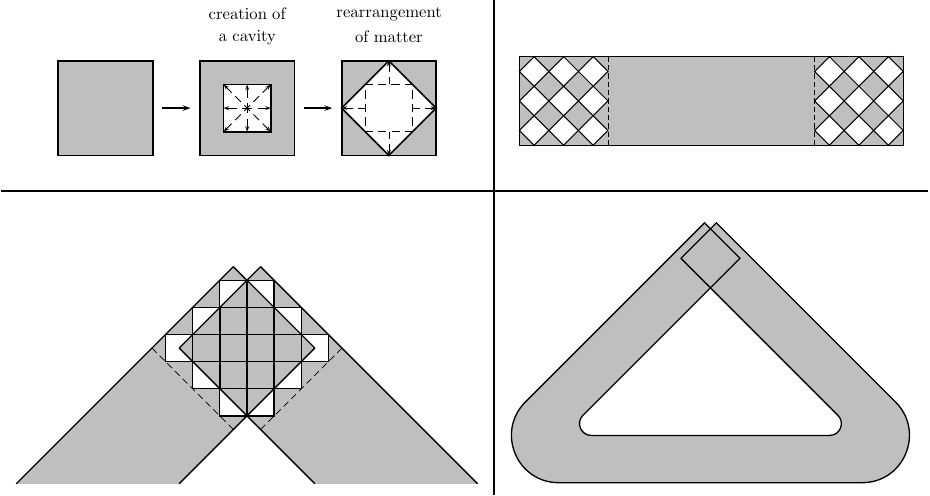}
   \end{center}
   \caption{\label{fig:prova}The pathological example by M\"uller and Spector. % In the upper left frame, the formation of a cavity and subsequent continuous deformation of a quadratic reference configuration is depicted. In the upper right frame, this building block is scaled and periodically continued to a larger square. Two of these larger squares are glued to the ends of a  strip. In the lower left frame, the strip is bent so that the material from one end fills the voids from the other. This map is invertible almost everywhere. Letting the scaling in step 2 go to zero, the deformations weakly converge to the one in the lower right frame. This map is 2-to-1 on a set of positive measure.
   }
\end{figure}

The domain of the example is a strip $\Omega\subset \R^2$. The deformations  are in $W^{1,p}(\Omega,\R^2)$ for all $p<2$, and are constructed as follows: % The strip is perforated at both its ends by  cavities. This works as follows:
One starts off with the formation of one single cavity in a quadratic reference configuration, and subsequent continuous deformation. This is depicted in the upper left frame in Figure~\ref{fig:prova}.
In the upper right frame, this building block is scaled and periodically continued to a larger square. Two of these larger squares are the end parts of the deformed rectangular strip $u(\Omega)$. 
Then the strip is bent so that material from one end covers the voids from the other (see the lower left frame of Figure~\ref{fig:prova}). The  map constructed in this way is invertible almost everywhere. Letting the period of the perforation tend to 0, the resulting sequence converges weakly in $W^{1,p}(\Omega,\R^2)$, for all $p<2$, to a deformation that is 2-to-1 on a set of positive measure (see the lower right frame of Figure~\ref{fig:prova}).

\subsection{Interpenetration for codimension one maps}

For maps $\R^{n-1}\supset U\to\R^{n}$ as they occur in  plate
theories, the above definition of invertibility almost everywhere is not suitable. Here,
modifications on sets of measure zero will be enough to make
deformations with interpenetration of matter injective. We by-pass this problem
by restricting
ourselves to continuous deformations -- in fact, we will even require Lipschitz
continuity, since this is general enough for all applications to the derivation
of plate theories.\\
%\item[3.] Self-penetration is a \emph{global} property. 
For $U\subset \R^{n-1}$, we let  $\hat U\subset\R^n$ denote the boundary of the cylinder over $U$:
\[
   \hat U:=\partial\left(U\times[0,1]\right)\,.
\]
In the following, we will identify $U$ with $U\times\{0\}\subset \hat
U$. 

% \begin{figure}[htb]
%    \begin{center}
%       \includegraphics[scale=.5]{tikz2}
%    \end{center}
%    \caption{The extension  $\hat u_1:\hat U_1\to \R^3$.\label{fig:2}}
% \end{figure}

% \begin{figure}[htb]
%    \begin{center}
%       \includegraphics[scale=.5]{tikz3}
%    \end{center}
% \caption{An example of simple interpenetration.\label{fig:3}}
% \end{figure}

\begin{figure}[htb]
   \centering
   \includegraphics{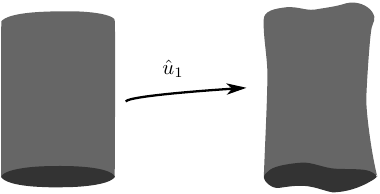}
\caption{The extension  $\hat u_1:\hat U_1\to \R^3$.\label{fig:2}}
\end{figure}

\begin{figure}[htb]
   \centering
   \includegraphics{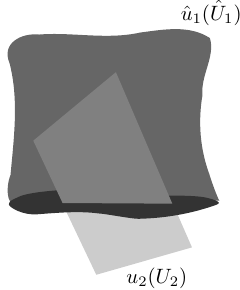}
   \caption{An example of  interpenetration.\label{fig:3}}
\end{figure}

\begin{definition}[Interpenetration]
\label{def:simple_interpenetration}
For $i\in\{1,2\}$, let $U_i\subset \R^{n-1}$ be simply connected
  Lipschitz domains and $u_i\in
{\mathrm{Lip}}(U_i,\R^{n})$. We say that $u_2$ interpenetrates $u_1$  if
there exists a Lipschitz-continuous extension $\hat u_1:\hat U_1\to\R^n$ of $u_1$ with the
following properties:
\begin{itemize} 
\item[(i)] The sets
\[
\begin{split}
\big\{x\in U_2:\,u_2(x)\not\in \hat u_1(\hat U_1),\,\deg(\hat u_1,\hat
U_1,u_2(x))=k\big\}\,,\quad k\in\N\,%% \\
 %% \big\{x\in U_2:&\,u_2(x)\not\in \hat u_1(\partial\hat U_1),\,\deg(u_2(x),\hat U_1,\hat u_1)\neq k\big\}
\end{split}
\]
have positive $\Lb^{n-1}$-measure for at least two different $k\in\N$.
\item[(ii)] The extension satisfies
\begin{equation}
\begin{split}
\hat u_1 ( \hat U_1\setminus U_1)\cap u_1(U_1)=&\emptyset\,,\\
\overline{\hat u_1 ( \hat U_1\setminus U_1)}\cap \overline{u_2(U_2)}=&\emptyset\,.
\end{split}
\label{eq:32}
\end{equation}
\end{itemize}
\end{definition}

We have depicted the extension $\hat u_1:\hat U_1\to\R^n$ in Figure~\ref{fig:2},
and the typical situation of  interpenetration in Figure~\ref{fig:3}.
% The typical situation of simply interpenetrating maps is sketched in Figure ... The reader can convince her/himself that  ``not too complicated''  cases of interpenetration are indeed covered by the definition. However, at first glance, it might seem as if it were unnecessarily complicated. This is not so, as we explain in the following remarks. 
\begin{example} %---{{{
   \label{example:1}
   Let $U_1=U_2=[0,1]^{2}$ and $v_i:C^{1}(U_i)$ for $i=1,2$. 
   Moreover, suppose that the set $A_{1}:=\insieme{x:\ v_{1}(x)<v_{2}(x)}$  and that $A_{2}:=\insieme{x:\ v_{2}(x)<v_{1}(x)}$ are both  open, non-empty and simply connected. 
   Set $u_{i}(x')=(x',v_{i}(x'))$ for $i=1,2$, and define the extension of $u_1$
   by $\hat u_1:\hat U_1=\partial [0,1]^3\to\R^3$,  
    $(x',x_{3})\mapsto(x',v_{1}(x')+x_{3}M)$, where $M:=\sup |v_{2}-v_{1}|$. 
   We will now check that $\deg(u_1,\hat{U}_{1},z)=0$ for $z\in A_{2}$ and that there exists $z\in A_{1}$ such that $\deg(u_1,\hat{U}_{1},z)>0$. 
   Indeed, there is an obvious way  to extend $\hat u_1$ to $[0,1]^3$:
\[
\hat u_1^*:[0,1]^3\to \R^3, \quad  (x',x_{3})\mapsto(x',v_{1}(x')+x_{3}M)\,.
\]
Trivially, $\hat u_1=\hat u_1^*$ on $\hat U_1$. Hence, $\deg(\hat u_1,\hat
U_1,\cdot)=\deg(\hat u_1^*,\hat U_1,\cdot)$. For $\hat u_1^*$, we may compute
the degree by (cf.~Section~\ref{sec:Brouwer})
   \begin{equation*} %---{{{
      \begin{split}
         \deg(\hat{u}_{1}^*,\hat{U}_1,z)= \sum_{x\in \hat{u}_1^{-1}(z)} \sgn(\det(\nabla \hat{u}_{1}^*(x)))\,.
      \end{split}
   \end{equation*} %---}}}
% By slight abuse of notation, we will just drop the ``$^*$'' from the notation,
% and identify $\hat u_1$ and $\hat u_1^*$ from now on.
 For every $z\in \R^{3}$, we  have that $\left(\hat{u}_1^*\right)^{-1}(z)$ is either empty or has one element. 
 In the latter case, it holds
 \begin{equation*} %---{{{
    \begin{split}
       \det(\nabla \hat{u}_{1}^*(x)) = \det \left( \begin{array}{cc} %---{{{
             \id &{\mathbf 0}\\
             (\nabla' v)^{T}&1
       \end{array} %---}}}
    \right) = 1,
    \end{split}
 \end{equation*} %---}}}
 where $\hat{u}_{1}^*(x)=z$. In particular for any $x\in 
 A_{2}$, one has that $\deg(\hat{u}_{1}^*,\hat{U}_{1},u_{2}(x))=0$. 
 On the other side, one easily sees that for every $x\in A_{1}$ one has that
 \begin{equation*} %---{{{
    \begin{split}
       \deg(\hat{u}_{1}^*,\hat{U}_{1},u_{2}(x)) =1.
    \end{split}
 \end{equation*}
Thus $u_2$ interpenetrates $u_1$.%---}}}
\end{example} %---}}}

\begin{remark} \ 
   \begin{enumerate}
      \item Definition~\ref{def:simple_interpenetration}  is  asymmetric with respect to $u_1$,
         $u_2$. This is done on purpose.  It is always possible to reverse the roles by shrinking the domain of $U_1$, but we are neither  going to prove nor use this fact.
         
      \item If $U$ is closed and $u:U\to \R^3$ is an embedding, then there do not exist disjoint subsets
         $U_1,U_2\subset U$ such that $u_2:=u|_{U_2}$ interpenetrates $u_1:=u|_{U_1}$. The converse is not true: there exist non-injective maps
         $u:U\to \R^3$  such that it is not possible to choose $U_1,U_2$ such
         that $u_2$ interpenetrates $u_1$ (defined as before). This is the
         case, e.g., if the graphs $u(U_1)$ and $u(U_2)$ touch, but do not intersect. 
         This is a  desirable feature of a definition for interpenetration of matter. 
         Indeed,  two surfaces that are touching but can be separated via infinitesimal perturbations, should not be considered as interpenetrating, as they can be approximated by recovery sequences with disjoint graphs. 

      \item Even though we chose the case of intersecting graphs to illustrate Definition~\ref{def:simple_interpenetration} in Example~\ref{example:1}, the definition is much more  flexible than that. 
        In particular, it is invariant under surface reparametrization. 
\item It might seem at first sight as if the requirement that the sets
\[
\big\{x\in U_2:\,u_2(x)\not\in \hat u_1(\hat U_1),\,|\deg(\hat u_1,\hat
U_1,u_2(x))|=k\big\}
\]
have positve measure for $k\in\{0,1\}$
would be equivalent to Definition \ref{def:simple_interpenetration} (i). However this would exclude cases such as the one depicted in Figure \ref{fig:curve} (where $n=2$), which are also covered by Definition \ref{def:simple_interpenetration} (i).
   \end{enumerate}

 \end{remark}

\begin{figure}[htb]
   \centering
   \includegraphics{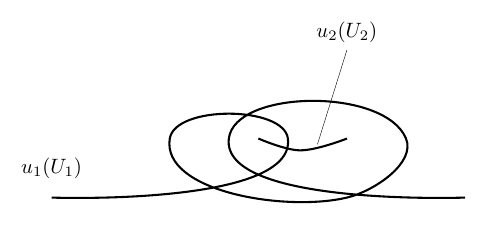}
   \caption{Interpenetrating curves with $\deg(\hat u_1,\hat U_1,u_2(x))\neq 0$ for all extensions $\hat u_1$ of $u_1$ and $x\in U_2$. \label{fig:curve}}
\end{figure}

% We recall a theorem from differential geometry,
% which is a special case of the fundamental theorem of Riemannian geometry (see
% \cite{CiarletDG}):
% \begin{thrm}
% \label{fundRiem}
% Let $U$ be connected and simply connected, $R=(R_{ijkl})_{i,j,k,l\in{1,\dots, n}}$ be the curvature tensor associated to
% the $C^2$-Riemannian metric $g$ on $U$.
% If $R=0$, then there exists a map $\varphi\in C^3(U,\R^n)$, unique up to a rigid
% motion,
% such that 
% \begin{equation}
% g=\nabla \varphi^T\nabla\varphi\,.  \non
% \end{equation}
% \end{thrm}

\subsection{Statement of the main theorem}

\label{sec:mainthm}
%\begin{equation}
%E_g(u)=\int_U \text{dist}^2(\nabla u,\mathcal{F}(x))\d x\label{Fdist}
%\end{equation}
Let $S\subset\R^2$ be open and bounded, and let $\Omega_h=S\times
(-h/2,h/2)$. We write $\Omega\equiv \Omega_1$. We will consider sequences of functions $z_h:\Omega_h\rightarrow
\R^3$. It is convenient to define them on the same domain by introducing
$y_h:\Omega\rightarrow \R^3$ via 
$y_h(x_1,x_2,x_3)=z_h(x_1,x_2,hx_3)$.
Also, we
introduce the scaled gradient
\begin{equation}
\nabla_h y=(\nabla' y,\frac{1}{h}\partial_3y)\,.\non
\end{equation}
%% and for $\gamma\geq 0$, the scalings in $x_3$ direction
%% \[
%% M_h^\gamma=\left(\begin{array}{ccc}1&0&0\\0&1&0\\0&0&h^{-\gamma}\end{array}\right)\,.
%% \]
%% % For  $u\in W^{1,2}(\Omega^h,\R^3), y\in W^{1,2}(\Omega,\R^3)$, we define 
%% % \begin{align}
%% % E(u)=&\int_{\Omega^h}\d x\, \text{dist}^2(\nabla u,SO(3))\non\\
%% % E^h(y)=&\int_{\Omega}\d x\, \text{dist}^2(\nabla_h y,SO(3))\non\,.
%% % \end{align}
%% In the rest of the paper, we will repeatedly assume 
%% \begin{equation}
%% \varphi\in C^3(S,\R^2)\label{eq:5}
%% \end{equation}
%% such that there exist $c_1,c_2>0$ with the property 
%% \begin{equation}
%%   \label{eq:4}
%% c_1|\xi|^2\leq |(\nabla\varphi)\,\xi|^2 \leq c_2|\xi|^2\quad\forall x\in S,
%% \xi\in\R^2\,.
%% \end{equation}
%% In the following, we write 
%% \begin{equation}
%% \begin{split}
%% \gt=&\nabla\varphi^T\nabla\varphi\,,\\
%% g(x',x_3)=&\left(\begin{array}{cc} \gt(x') &  0\\0 &
%%     1\end{array}\right)\label{metricansatz}\,.
%% \end{split}
%% \end{equation}

\begin{theorem}
   \label{mainthm}
   Let $S$, $\Omega_h$ and $\Omega$ be as above, and let $U_1,U_2\subset S$ be disjoint simply connected Lipschitz sets.  Let  $u_1:U_1\to\R^3$,
   $u_2:U_2\to\R^3$ be Lipschitz and let $u_2$ interpenetrate $u_1$, let $\e>0$ and  $y_h$  a
   sequence in $W^{1,2}(\Omega,\R^3)$ such that   
   \begin{equation}
      \|\dist(\nabla_hy_h,{\mathrm{SO}}(3))\|_{L^2(\Omega)}^2<C h^{1+\e}\label{eq:8}
   \end{equation}
   and 
   \begin{align}
      \fint_{-1/2}^{1/2}  y_h(\cdot,x_3)\d x_3\wto u_i \quad\text{in }W^{1,2}(U_i,\R^3)\text{as }h\to 0\text{ for }i=1,2\,.\label{eq:7}
      %% \fint_{-1/2}^{1/2}  y_h(\cdot,x_3)\dx_3&\to \phi \quad\text{in
      %% }W^{1,2}(U_2,\R^3)\text{ as }h\to 0\,.\label{eq:39}
   \end{align}
   Then, for $h$ small enough, $y_h$ is not invertible almost everywhere.
\end{theorem}

\noindent
% {\bf Remarks.}\\
\begin{remark}
The crucial assumption here is \eqref{eq:8}. This condition (or, more precisely, its
2-dimensional analog)  is not fulfilled by the pathological examples from \cite{MR1346364}, whereas it does hold true for recovery sequences in the derivation of plate theories by $\Gamma$-convergence.
\end{remark}
% 3. If one allows for passing to a subsequence, the assumption of the existence
% of a $\phi$ such that \eqref{eq:39} holds is
% redundant. This fact follows from \eqref{eq:8}, but for the sake of brevity we
% do not prove this here and include it in the assumptions.

% For  $u\in W^{1,2}(\Omega^h,\R^3), y\in W^{1,2}(\Omega,\R^3)$, we define 
% \begin{align}
% E_g(u)=&\int_{\Omega^h}\d x\, \text{dist}^2(\nabla u,\mathcal{F}(x))\non\\
% E^h_g(y)=&\int_{\Omega}\d x\, \text{dist}^2(\nabla_h y,\mathcal{F}(x))\non\,.
% \end{align}
% We cite Theorem 2.6 (iii) of \cite{LP1}:
% \begin{thrm}
% \label{incompstrain}
% If the Gaussian curvature of $g$ is non-zero, then
% \begin{equation}
% \inf_{u\in W^{1,2}(U,\R^n)}\frac{1}{h^2}E^h_g(u)\geq c >0 \,.
% \end{equation}
% \end{thrm}

\section{Preliminaries}
\label{setting}

For $A\subset\R^n$, we recall  the definitions of $m$-dimensional Hausdorff and spherical
Hausdorff pre-measures and of the ``packing measure'',
\begin{align*}
   \H^m_\delta(A)=&\inf\bigg\{\omega(m)\sum_j 2^{-m}\diam(A_j):A\subset \cup_jA_j,\, \diam(A_j)/2\leq\delta\bigg\}\\
   \S^m_\delta(A)=&\inf\bigg\{\omega(m)\sum_j r_j^m:A\subset \cup_jB(x_j,r_j),\, r_j\leq\delta\bigg\}\\
   \mathcal P_{\delta}^m(A)=& \,\omega(m) \delta^m \inf \bigg\{ \# \{B(x_i,\delta)\}:\ 
   \bigcup_{i} B(x_i,\delta) \supset A \bigg \}
\end{align*}
where $m\in[0,\infty)$.  In the above definition, we also allow $\delta=\infty$.\\
It is well known (see e.g. \cite{MR0257325}) that the limits $\lim_{\delta\to
  0}\H^m_\delta$, $\lim_{\delta\to 0}\S^m_\delta$ define Borel
measures $\H^m,\S^m$ on $\R^n$, and that there exists a numerical constant $C=C(n)$ such
that
\[
   C^{-1}\S^m_\delta(A)\leq \H^m_\delta(A)\leq C\S^m_\delta(A)\quad \text{and} \quad
   \Pcal_{\delta}^m \geq \S^m_{\delta}(A) \geq  \H^{m}_{\infty}
\]
for every $A\subset \R^n$. Also, we recall the definition of the 1-capacity of a
set $A\subset\R^n$,
\[
   \kap_1(A)=\inf\{\text{Per}(E):\ E \text{ is an open set of finite perimeter and } A\subset E\}\,.
\]
From these definitions, it is easily seen that there exists a constant $C=C(n)$
with the propert
\begin{equation}
  \label{eq:35}
  \kap_1(A)\leq C \H^1_\infty(A)\quad\text {for all } A\subset \R^n\,.
\end{equation}
We cite the relative isoperimetric inequality for sets of finite perimeter. In
the following statement, for a set of finite perimeter $E$, $\partial_* E$
denotes the reduced boundary of $E$ (see~\cite{MR1014685}).
\begin{theorem}[{\cite{MR1014685}}]
   \label{thm:isoperimetric_ineq}
   Let $U\subset\R^n$ be a bounded open set with Lipschitz boundary. Then there exists a
   constant $C=C(U)$ such that for every set $E\subset \R^n$  of finite perimeter, 
   \begin{equation} %---{{{
      \label{eq:isoperimetric_ineq_local}
      \begin{split}
         \min\left\{|E\cap U|,|U\setminus E|\right\}^{n-1/n}
         \leq C \H^{n-1}(\partial_* E\cap U)\,.
      \end{split}
   \end{equation} %---}}}
   The same inequality holds true if one considers instead of $U$ the whole $\R^{n}$. Namely, there exists a constant $C=C(n)$ such that 
   \begin{equation} %---{{{
      \label{eq:isoperimetric_global}
      \begin{split}
         \min\left\{|E|,|\R^{n}\setminus E|\right\}^{n-1/n}
         \leq C \H^{n-1}(\partial_* E)\,.
      \end{split}
   \end{equation} %---}}}
\end{theorem}

Using the previous theorem, we will now prove a version of the isoperimetric 
inequality involving capacities instead of the Hausdorff measure.  Note that 
due to $\kap_{1} \leq C \hausd^{d-1} $ the next lemma is stronger than Theorem~\ref{thm:isoperimetric_ineq}. 
\begin{lemma}
   \label{lemma:isoperimetric-capacity}
   Let $U$ be a bounded open set with Lipschitz boundary. Then there exists a constant $C=C(U)$ such that for every bounded set $E$ of finite perimeter, 
		\begin{equation}
         \label{eq:isoperimetric_cap}
			\begin{split}
            \left(\min \left(|E\cap U |, |U\setminus E|\right)\right)^{(n-1)/n}	\leq C 
            \kap_1(\partial_{*} E\cap U)\,.
			\end{split}
		\end{equation}
      % The same inequality holds true if one considers instead of $U $ the whole $\R^{n} $. Namely, there exists a constant $C=C(n) $ such that 
      % \begin{equation} %---{{{
      %    \label{eq:isopermetric_cap_bis}
      %    \begin{split}
      %       \min\{|E |,|\R^{n}\setminus E |\}^{(n-1)/n}\leq C \kap_{1}(\partial_{*}E).
      %    \end{split}
%      \end{equation} %---}}}
\end{lemma}
\begin{proof}

   Suppose that the claim of the lemma were not true. Then there exists a sequence of sets $E_{k}\subset \R^{n}$ such that 
		\begin{equation}
         \label{eq:lemma_isopercapacity1}
			\begin{split}
            \left(\min \left(|E_{k}\cap U |, |U\setminus E_k|\right)\right)^{(n-1)/n}	\geq k 
            \kap_1(\partial_{*} E_k\cap U)\,.
			\end{split}
		\end{equation}

      Let us split the proof in two cases: either there exists an $M$ and a sequence $\insieme{E_{k}}$ such that 
     \begin{equation*} %---{{{
        \begin{split}
           \min(|E_{k}\cap U|, |U\setminus E_{k}|)^{(n-1)/n} \geq M
        \end{split}
     \end{equation*} %---}}}
      or for every sequence such that \eqref{eq:lemma_isopercapacity1}  holds, one has that 
      \begin{equation*} %---{{{
         \begin{split}
             \min(|E_{k}\cap U|, |U\setminus E_{k}|)^{(n-1)/n} \downarrow 0\text{  as }k\uparrow \infty. 
         \end{split}
      \end{equation*} %---}}}

      To deal with the first case, we will show that  it is not possible to have $\min(|E\cap U|, |U\setminus E|) > M$ and $\kap_{1} (\partial_{*}E \cap U) < \varepsilon $, with $\varepsilon $ suitably small. To show that also the second case leads to a contradiction,  we will use a spatial scaling and basically reduce  ourselves to having
      \begin{equation*} %---{{{
         \begin{split}
            \min(|E\cap U |,|U\setminus E|)=1.
         \end{split}
      \end{equation*} %---}}}

      \begin{itemize} %---{{{
         \item[Case 1.] There exists an $M$ and subsequence $\insieme{E_{k}}$ such that $\min(|E_{k}\cap U|, |U\setminus E_{k}|)^{(n-1)/n} \geq M$. The boundedness of $U$ implies that $\kap_{1}(\partial_{*}E_k \cap U)\leq \frac{|U|^{n/(n-1)}}{ k}$.  In particular, one has that $\kap_{1}(\partial_{*} E_{k}\cap U)\downarrow 0$.  %% Moreover, without loss of 
            %%    generality we will assume that $|E | \leq |U\setminus E | $. 
            Fix $\varepsilon>0$ sufficiently small. By the definition of 1-capacity, there exists an open set of finite perimeter  $V_{k}$ such that $\partial_*E_{k}\cap U\subset V_{k}$ and $\per(V_{k})\leq \kap_1(\partial_{*}E_{k} \cap U) +\varepsilon \leq 2\varepsilon$.   
            Using  the second part of Theorem~\ref{thm:isoperimetric_ineq}, one has that 
            $|V_{k}|\leq C\varepsilon^{n/(n-1)}$.  Let us denote by
            $\tilde{E}_{k}:=E_{k}\cup V_{k}$.   We claim that
            \begin{equation}
               \partial_{*} \tilde{E}_{k} \cap U\subset \partial_{*}V_{k}\,.\label{eq:19}
            \end{equation}
            Indeed, note that $\partial_{*}(E _{k} \cup V_{k})\cap U \subset (\partial
            _{*}E_{k}\cup \partial_{*}V_{k})\cap U$. By $\partial_{*}E_{k}\cap U \subset
            V_{k}$, one has that every $x\in \partial_{*}E_{k}$ is an interior point (and in
            particular a set of $1$-density, see \cite{MR1014685}), and thus $x\not \in \partial_{*}(E_{k}\cup V_{k})\cap U$ which proves \eqref{eq:19}.

            Hence,
            \begin{equation}
               \label{eq:21} %---{{{
               \begin{split}
                  \min \left(|E_{k}\cap U|,|U\setminus E_{k}|\right)&\leq  C \min
                  \left(|\tilde{E_{k}}\cap U|,|U\setminus \tilde{E}_{k}|\right)+C\varepsilon^{n/(n-1)}\\
                  &\leq   C(U)\left( \left(\hausd^{n-1}(\partial_* \tilde{E}_{k}\cap U_{k})\right)^{n/(n-1)}+\varepsilon^{n/(n-1)}\right) \\
                  &\leq  C(U)\left(\left(\per(V_{k})\right)^{n/(n-1)}+\varepsilon^{n/(n-1)}\right) \\
                  &\leq  C(U)\left(\varepsilon^{n/(n-1)}\right),
               \end{split}
            \end{equation} %---}}}
            where in the second inequality above, we have used Theorem~\ref{thm:isoperimetric_ineq}. By the arbitrariness of $\varepsilon $, we obtain a
            contradiction.  

         \item[Case 2.] Let us now suppose that for every sequence ${E_{k} }$ such that \eqref{eq:lemma_isopercapacity1} holds, one has that $\min(|E_{k}\cap U |, |U\setminus E_{k} |)\downarrow 0$. 
            Without loss of generality we may assume that $\min(|E_{k}\cap U|, |U\setminus E_{k}|) = |E_{k}\cap U|$. 
            Note that both sides of \eqref{eq:isoperimetric_cap} have the same spatial scaling. 
            Thus, there exists $\lambda_{k} >0$ such that $|\lambda_{k}E_{k}| =1$ and
            \begin{equation*} %---{{{
               \begin{split}
                  |\lambda_{k}E_{k} |^{(n-1)/n} \geq k \kap_1( (\lambda_{k}\partial_{*}E\setminus)\cap (\lambda_{k}U)).
               \end{split}
            \end{equation*} %---}}}
            Hence, by rescaling  by $\lambda_{k}$ one can assume without loss of generality that $|E_{k}|=1$ and $U_{k}\uparrow \R^{d}$, where $U_{k}:= \lambda_{k}U$. 

            After this observation the proof will proceed in a similar fashion as in the first case. 
            As in the previous case, one has that $\kap_{1}(\partial_{*}(E_{k}\cap U_{k}))\downarrow0$. 
            Using the definition of $1$-capacity, there exists an open set of finite perimeter $V_{k}$ such that $\partial_*(E_{k}\cap U_{k})\subset V_{k}$ and $\per(V_{k})\leq \kap_1(\partial_{*}(E_{k}\cap  U_{k}))+\varepsilon \leq 2\varepsilon$.   
            Without loss of generality, we may assume additionally that $|V_{k}|<+\infty$. Indeed, let $B\supset U$ be a ball, and $B_{k}:=\lambda_{k}B $. Because $B_{k} $ is convex one has that $\per(B_{k}\cap V_{k})\leq \per(V_{k})$ thus by taking $V_{k}\cap B_{k}$ instead of $V_{k}$, one has the additional requested property. \\
            Using the second part of Theorem~\ref{thm:isoperimetric_ineq}, one has that $|V_{k}|\leq C \varepsilon^{n/(n-1)}$. Denote by $\tilde{E}_{k}:= E_{k}\cup V_{k}$ and notice as before that $\partial_{*}\tilde{E}_{k}\subset \partial_{*} V_{k}$. 
            Hence, following exactly the same chain of inequalities as in \eqref{eq:21},
            %    \begin{equation*} %---{{{
            %    \begin{split}
            %       \min \left(|E_{k}\cap U_{k}|,|U_{k}\setminus E_{k}|\right)^{(n-1)/n}&\leq  C \min
            %       \left(|\tilde{E_{k}}\cap U_{k}|,|U_{k}\setminus \tilde{E}_{k}|\right)^{(n-1)/n}+C\varepsilon^{n/(n-1)}\\
            %       &\leq  C(U)\left( \hausd^{n-1}(\partial_* \tilde{E}_{k}\cap U_{k})+\varepsilon^{n/(n-1)}\right) \\
            %       &\leq C(U)\left(\per(V_{k})+\varepsilon^{n/(n-1)}\right) \\
            %       &\leq C(U)\left(\varepsilon+\varepsilon^{n/(n-1)}\right),
            %    \end{split}
            % \end{equation*} %---}}}
            this gives a contradiction as before. 
      \end{itemize} %---}}}
      % Inequality \eqref{eq:isopermetric_cap_bis} follows in the same way as above.  
\end{proof}

\subsection{Miscellaneous results from the literature}

In the proof of our main theorem, we will use the following geometric rigidity result. 
\begin{theorem}[{\cite[Theorem~3.1]{MR1916989}}]
   \label{thm:rigidity}
   Let $U\subset \R^n$ be a bounded Lipschitz domain, with $n\geq 2$. Then there exists a constant $C=C(U)$ with the following property:
   For every $v\in W^{1,2}(\R^n)$, there is an associated rotation $R\in \son$
   such that,
   \begin{equation*}
      \begin{split}
         \|\nabla v -R \|_{L^2(U)}\leq C \|\dist(\nabla v,\son)\|_{L^2(U)}
      \end{split}
   \end{equation*}
   The constant $C(U)$ is invariant under rescaling of the domain.
\end{theorem}

We will also use Zhang's Lemma \cite{MR1205403}. An inspection of its
proof in the latter reference shows that the following (slightly modified) statement holds true as well. 
\begin{theorem}[\cite{MR1205403}, Lemma 3.1]
   \label{thm:zhang}
   Let $K>0$. 
   There exist constants $C_1=C_1(n,m)$, $C_2=C_2(n,m,K)$ with the following property: 
   If $U\subset \R^n$ is open and bounded, $f\in W^{1,1}(U,\R^m)$ and $\e>0$ such that
   \[
      \int_{U\cap\{|\nabla f|\geq K\}}|\nabla f|\d x <\e \,,
   \]
   then there exists $\tilde f\in W^{1,\infty}(U,\R^m)$ such that
   \begin{align*}
      %% f=& \,\tilde f \quad\text{ on }\quad U\cap\{|\nabla f|\leq K\}\\
      \|\nabla \tilde f\|_{L^\infty(U)}\leq & \,C_1K\,,\\
      {\mathcal L}^n\left(\{x:f(x)\neq \tilde f(x)\}\right)\leq& \,C_2 \e\,.
   \end{align*}
\end{theorem}

\section{Proof of Theorem~\ref{mainthm}}
\label{sec:pf_mainthm}

Let $S$, $\Omega$ and $\Omega_h$ be as defined in Section~\ref{sec:mainthm}.
%TODO comment before
Our strategy is as follows.
In Proposition~\ref{prop:05291369832965} below, we will consider maps $y_{h}$ on a $3$-dimensional domain. 
We reduce the domain to $2$ dimensions and assume that the thus obtained maps
are $2$-to-$1$ on a ``large set'' in terms of capacity, and we will show that
this is sufficient to conclude that the maps $ y_h$ are $2$-to-$1$ on a set
whose $\Lb^3$-measure is of order $h^2$. This will be the main step in the proof
of Theorem~\ref{mainthm}. In the proof of the proposition, we
will need the following lemma:

\begin{lemma}
\label{lemma:morrey}
Let $\e>0$,  $x_h \in S$, $\alpha\in(0,1/2]$, and $y_h\in
W^{1,\infty}(\Omega_h;\R^3)$ with 
\[
\begin{split}
  \int_{B(x_h,\alpha h)}\dist^2(\nabla y_h,{\mathrm{SO}(3)})\d x\leq &C h^{3+\e}\,,\\
  \|\nabla y_h\|_{L^\infty}\leq &C\,.
\end{split}
\]
Then there exist  rigid motions $A_h:\R^3\to\R^3$ such that
\begin{equation}
\sup_{x'\in B(x_h,\alpha h)} |y_h(x')-A_h(x')|\leq C h^{1+\bar \e}.\label{eq:17}
\end{equation}
where  $\bar \e=\e/(3+\e)$.
\end{lemma}
\begin{proof}
       By Theorem~\ref{thm:rigidity}, there exists a numerical constant $C=C(C_1,C_2)$ and
     $R_h\in {\mathrm{SO}(3)}$ such that
    \[
    \int_{B(x_h,\alpha h)} |\nabla y_h-R_h|^2\dx \leq C h^{3+\e}\,.
    \]
    We set
    \begin{equation*}
       b_h=\fint_{B(x_h,\alpha h)}\left(y_h(x)-R_hx\right)\dx\,.
    \end{equation*}
    By the Poincar\'e inequality, there exists $C=C(C_1,C_2)$ such that
    \begin{equation}
    \int_{B(x_h,\alpha h)} |y_h(x)-R_hx-b_h|^2\d x\leq C h^{5+\e} \,.\label{eq:10}
    \end{equation}
    Let $A_h$ be the rigid motion
    \begin{equation*}
       x\mapsto R_hx+b_h\,.
    \end{equation*}
    Since $\|y_h-A_h\|^2_{W^{1,2}(B(x_h,\alpha h))}\leq C  h^{3+\e}$, and
    $\|y_h-A_h\|^2_{W^{1,\infty}(B(x_h,\alpha h))}\leq C$, we also have
    \[
    \|\nabla(y_h-A_h)\|^p_{L^p(B(x_h,\alpha h))}\leq C(p) h^{3+\e}
    \]
    for all $p\in[2,\infty)$. \\
Let $w_h=y_h-A_h$, and $B=B(x_h,\alpha h)$. 
Using \eqref{eq:10} and H\"older's inequality, we have
\begin{equation*}
\begin{split}
\fint_B |w_h|&\leq\frac{1}{\omega(3)(h/2)^3}\left(\int_B |w_h|^2\right)^{1/2}(\omega(3)(h/2)^3)^{1/2}\\
&\leq C h^{(5+\e)/2}\,.
\end{split}
\end{equation*}
We set $p=3+\e$. For $x\in B$,  we have the following estimate (which is used in a similar fashion in the proof of Morrey's Inequality, see e.g.~the proof of the latter in \cite{gilbarg2001elliptic})
\begin{equation*}
\begin{split}
\fint_B|w_h(x)-w_h(z)|\d z\leq & C\,\int_B \frac{|\nabla w_h(z)|}{|x-z|^2}\d z\\
\leq & C \left(\int_B |\nabla w_h|^p\right)^{1/p}\left(\int_B |x-z|^{-2p/(p-1)}\right)^{(p-1)/p}\\
\leq & C h^{(3+\e)/p}h^{1-3/p} \\
\leq & C h^{1+\e/(3+\e)}\,.
\end{split}
\end{equation*}
Thus we get
\begin{equation}
\label{eq:whoh}
\sup_{x\in B}|w_h(x)|\leq  \fint_{B} |w_h|\d z
+\sup_{x\in B}\fint|w_h(x)-w_h(z)|\d z\leq C h^{1+\bar \e}
\end{equation}
which proves (\ref{eq:17}).
\end{proof}

%TODO change the proof
\begin{proposition} %---{{{
   \label{prop:05291369832965}
   Let $y_h:\Omega_h\to\R^3$ be  Lipschitz, and let $C^*,\e>0$ such that
   \begin{equation*} %---{{{
      \begin{split}
         \int_{\Omega_h}\dist^{2}(\nabla y_h,{\mathrm{SO}(3)}) \leq & \,C\,h^{2+\e}\,,\\
           \qquad \|\nabla {y}_h \|\leq & \,C^*\,.
      \end{split}
   \end{equation*} %---}}}
   Further, with $ u_h(\cdot) =y_h(\cdot,0) $, and
\[
F_h:=\left\{ x: \text{ there exists }\bar x\in S \text{ s.t.\ } u_h(x)=u_h(\bar x)
  \text{ and } |x-\bar x|>2h\right\}
\]
   assume 
that 
\[
    \kap_1\big( F_h\big) \geq C_1\,\quad\text{ for all }h<h_0
\]
for some constants $C_1,h_0>0$.
   Then there exists $c=c(C_1)>0$ such that for $h$ small enough,
   \begin{equation} %---{{{
      \label{eq:measure_1-to-1}
      \begin{split}
         \leb^3 \left(\left\{x:\, y_h \text{ is not $1$-to-$1$ at $x $}\right\}\right) > ch^{2}\,.
      \end{split}
   \end{equation} %---}}}
\end{proposition} %---}}}

\begin{proof} %---{{{
   % By the assumption on the size of $|\{ y_h\neq \tilde{y}_h \} |$, it is clear that is enough to show the claim for $\tilde y_h$.\\
   {\bf Step 1.} Covering of $F_h$ by balls of size $h$ and definition of auxiliary
   partitions. For simplicity let us denote
   \begin{equation*} %---{{{
      \begin{split}
         E_h:= &\int_{\Omega_h}\dist^{2}(\nabla y_h,{\mathrm{SO}(3)})\,. %\\ 
      \end{split}
   \end{equation*} %---}}}
   Fix a set of points $X_h=\{x_i\}_{i\in I}\subset F_h$ such that 
   $F_h\subset \cup_I B(x_i,h/2)$ and $B(x_i,h/10)\cap B(x_j,h/10)=\emptyset$ for $i\neq j$. 
   Such a set $X_h$ exists by Vitali's Covering Lemma. 
   By definition of $F_h$, for every $x\in X_h$, there exists $\bar x\in F_h$ such that $u_h(x)=u_h(\bar x)$ and $B(x,h/2)\cap B(\bar x,h/2)=\emptyset$. \\ 
   In the following, we identify the points $x\in X_h\subset S$ with the points
   $(x,0)\in S\times\{0\}\subset\Omega_h$, and whenever we speak of a ball around a
   point $x\subset X_h$, it is understood to be three-dimensional.\\ 
   \newcommand{\Xl}{X_h^{\mathrm{low}}} 
   \newcommand{\Xh}{X_h^{\mathrm{high}}} 
   \newcommand{\Xlb}{\bar X_h^{\mathrm{low}}} 
   \newcommand{\Xhb}{\bar X_h^{\mathrm{high}}} 
   Now we introduce several useful partitions of $X_h$. First, we define the set of 
   $x\in X_h$ with ``low energy'',
   \begin{equation}
      \Xl:=\left\{x\in X_h:\int_{B(x,h/10)} \dist^{2}(\nabla y_h,{\mathrm{SO}(3)}) \leq \frac{4hC_2}{C_1} E_h\right\}\,,\label{eq:28}
   \end{equation}
   where $C_2$ is the constant from \eqref{eq:35} with $n=3$.
   The complement (the set of $x\in X_h$ with ``high energy'') is denoted by
   $\Xh=X_h\setminus \Xl$.\\
   Secondly, for $x\in X_h$, we write $M(x):=B(x,h/(20 C^*))\cup B(\bar
   x,h/(20C^*))$. We define the set of $x$ with ``low pair-energy'' as
   \begin{equation}
      \Xlb:=\left\{x\in X_h:\int_{M(x)} \dist^{2}(\nabla y_h,{\mathrm{SO}(3)}) \leq \frac{4hC_2}{C_1}
         E_h\right\}\,.\label{eq:29}
   \end{equation}
   The complement (the set of $x\in X_h$ with ``high pair-energy'') is denoted by
   $\Xhb=X_h\setminus \Xlb$.\\
   Finally, we introduce the partition $X_h=\G\cup\B$ where we call
   $\G$ the set of ``good'' points and $\mathcal B$ the set of ``bad'' points.
   We define the set of ``good'' points as the union $\G=\G^1\cup \G^2$, where the
   latter are defined as follows,
   \begin{equation}
      \label{eq:24}
      \G^1=\left\{x\in \Xl:\exists x'\in \Xl,\, x\neq x', \text{ with }|y_h(x)-y_h(x')|\leq h/10\right\}\,,
   \end{equation}
   and
   \begin{equation}
      \label{eq:30}
      \G^2= \Xlb\,.
   \end{equation}
   Now we claim that
   there exists a constant
   $C_3>0$ such that for $h$ small enough,
   \begin{equation}
      \label{eq:13}
      \#\G>C_3h^{-1}
   \end{equation}
   and
   \begin{equation}
      \label{eq:12}
      \leb^3(\{x'\in B(x,h/10):y_h \text{ is not $1$-to-$1$ at
      }x'\})>C_3 h^3\,\quad\text{ for all $x\in\G$.}
   \end{equation}
   This will be enough to prove the proposition since the balls of radius
   $h/10$ and centers in $\G$ are mutually disjoint.\\
   {\bf Step 2.} Proof~of~\eqref{eq:13}.
   Recalling the relations between capacities and Hausdorff pre-measures we 
   have $ \kap_1 \leq C_2 \hausd^{1}_{\infty}\leq C_2 \mathcal{P}^{1}_{h/2} $ for some numerical constant $C_2$. 
   Hence $\mathcal{P}^{1}_{h/2}(F_h)\geq C_1C_2^{-1} $, and in particular for every
   covering of $F_h$ with balls $\insieme{B_i}$ of radius $h/2$ we have that 
   \begin{equation} 
      \label{eq:15} \begin{split}
         2\sum_{i} r(B_i) \geq C_1C_2^{-1}, 
      \end{split}
   \end{equation} 
   where $r(B)$ denotes the radius of the ball $B$.
   Applying~\eqref{eq:15} to the cover by balls with centers in $X_h$ constructed above, we get
   \begin{equation}
      \# X_h\geq \frac{1}{h} \mathcal{P}_h^1(F_h)\geq \frac{C_1}{hC_2}\,.\label{eq:23}
   \end{equation}
   By definition, $\B=X_h\setminus \G$, and hence
   \begin{equation}
      \label{eq:31}
      \B=\Bigg\{x\in \Xhb: \text{Either }\left(x\in \Xh \right)\text{ or }\left(\not \exists x'\in
         \Xl,\,x\neq x', \text{
            with }|y_h(x)-y_h(x')|\leq h/10\right)\Bigg\}\,.
   \end{equation}
   Hence we have $\B\subset \B^1\cup \B^2$ with
   \begin{equation} %---{{{
      \begin{split}
         \B^1=\Xh=\left\{x\in X_h: \int_{B(x,h/10)} \dist^{2}(\nabla y_h,{\mathrm{SO}(3)}) \geq \frac{4hC_2}{C_1} E_h\right\}
      \end{split}\,,
      \label{eq:gooddef} 
   \end{equation} %---}}}
   and
   \begin{equation}
      \label{eq:20}
      \B^2=\Bigg\{x\in \Xl:\left(x\in \Xhb\right)% \int_{M(x)} \dist^{2}(\nabla y_h,{\mathrm{SO}(3)}) \geq \frac{4hC_2}{C_1}
      % E_h
      \text{ and }
      \left(\not \exists x'\in \Xl,\, x'\neq x,\, \text{ with
         }
         |y_h(x)-y_h(x')|\leq h/10\right)\Bigg\}
   \end{equation}
   By~\eqref{eq:gooddef} and the fact that the $h/10$-balls with centers in $X_h$ are
   mutually disjoint, we have
   \begin{equation}
      \# \B^1\leq \frac{C_1}{4C_2 h}.\label{eq:22}
   \end{equation}
   For $x_1,x_2 \in \B^2$, we have 
   \begin{equation}
      \begin{split}
         |\bar x_1-\bar x_2 |\geq &\frac{1}{C^*}|y_h(\bar x_1)-y_h(\bar x_2)|\\
         =&\frac{1}{C^*}|y_h( x_1)-y_h( x_2)|\\
         \geq& \frac{h}{10C^*}\,,
      \end{split}\label{eq:14}
   \end{equation}
   and hence the balls $B(\bar x,h/(20C^*))$ with $x\in \G^2$ are mutually
   disjoint.
   By the definition of $\Xhb$, this implies
   \begin{equation}
      \#\B^2\leq \frac{C_1}{2 C_2h}\,.\label{eq:16}
   \end{equation}
   Combining~\eqref{eq:23}, \eqref{eq:22} and \eqref{eq:16}, we have proved 
   \eqref{eq:13} for $C_3\leq\frac{C_1}{4C_2}$. \\
   {\bf Step 3.} Proof of \eqref{eq:12} for $x\in \G^1$. Let $x\in \G^1$. By the
   definition of $\G^1$ in \eqref{eq:24}, there exists
   $x'\in \G^1$, $x'\neq x$, with $|y_h(x)-y_h(x')|\leq h/10$.
   Let $B_h^x=B(x,h/10)$, $B_h^{x'}=B(x',h/10)$.
   The conditions of Lemma~\ref{lemma:morrey} with $\alpha=1/10$ are fulfilled for
   $y_h$ on both of these balls, and hence we obtain the existence of rigid motions
   $A_h^x,A^{x'}_h$ (depending on $x,x',h$) that satisfy
   \begin{equation}
      \sup_{z\in B_h^x}|y_h(z)-A_h^x(z)|\leq C h^{1+\bar\e}\,,\quad
      \sup_{z\in B_h^{x'}}|y_h(z)-A^{x'}_h(z)|\leq C h^{1+\bar\e}\,.\label{eq:33}
   \end{equation}
   Note that $C$ is independent of $x,x'\in\G^1$ and of
   $h$. 
   The images of $B_h^x$ and $B_h^{x'}$ under $A_h^x$ and $A^{x'}_h$ respectively
   are balls of radius $h/10$ and centers $A_h^x(x),A^{x'}_h(x')$.
   By~\eqref{eq:33},
   \begin{equation}
      \limsup_{h\to 0}\inf_{x\in \G^1}\frac1h|A_h^x(x)-A^{x'}_h(x')|\leq 1/10\,.
      \label{eq:38}
   \end{equation}
   Set
   \[
      c_0=\frac{\leb^3(B(0,1)\cap B(e_1,1))}{\leb^3(B(0,1))}\,.
   \]

   By~\eqref{eq:38}, given $\delta>0$, we may choose $h_1=h_1(\delta)$ such that for all
   $h<h_1$, there exists a set $W_h\subset B_h^x$ with
   \begin{align}
      \label{eq:37}
      \leb^3(W_h)\geq &(c_0-\delta)\leb^3(B_h^x)\\
      \label{eq:40}A_h^x(W_h)\subset &A_h^{x'}(B_h^{x'})\\
      \label{eq:41}\dist(A_h^x(W_h),A_h^{x'}(\partial B_h^{x'}))\geq &\frac{\delta h}{C}\,.
   \end{align}
   In particular, \eqref{eq:40} implies
   \begin{equation}
      \label{eq:42}
      \deg(A_h^{x'},\partial B_h^{x'},A_h^x(z))=1\quad\text{for }z\in W_h\,.
   \end{equation}
   We define homotopies $H_h^x:[0,1]\times \overline{B_h^x}\to\R^3$,
   $H_h^{x'}:[0,1]\times \overline{B_h^{x'}}\to\R^3$ by 
   \begin{equation}
      \label{eq:43}
      \begin{split}
         H^x_h(t,z)=&t y_h(z)+(1-t)A_h^x(z)\\
         H^{x'}_h(t,z)=&t y_h(z)+(1-t)A_h^{x'}(z)\,.
      \end{split}
   \end{equation}
   By \eqref{eq:33} and \eqref{eq:41}, we have (for $h$ small enough) 
   \[
      H^{x}_h(t,z)\not\in H^{x'}_h(\partial B_h^{x'}) \quad\text{ for }
      t\in[0,1],\,z\in W_h\,.
   \]
   By \eqref{eq:111} and \eqref{eq:42}, this yields
   \begin{equation}
      \label{eq:42}
      \begin{split}
         \deg(A_h^{x'},\partial B_h^{x'},A_h^x(z))=&\deg(H_h^{x'}(0,\cdot),\partial
         B_h^{x'},H_h^x(0,z))\\
         =&\deg(H_h^{x'}(1,\cdot),\partial B_h^{x'},H_h^x(1,z))\\
         =& \deg(y_h|_{B_h^{x'}},\partial B_h^{x'},y_h(z))=1\quad\text{for }z\in
         W_h\,.
      \end{split}
   \end{equation}
   By \eqref{eq:37} and the arbitrariness of $\delta$, this implies
   % which implies, by \eqref{eq:33} and the continuity of the mapping degree
   % with respect to uniform convergence,
   % \begin{equation*}
   % \lim_{h\to 0} \inf_{x\in \G^1}\frac{\leb^3\left(
   % \left\{z\in B_h^x:\deg(A^{x'}_h ,\partial B_h^{x'},A_h^x(z))=1\right\}
   % \right)}{\leb^3(B_h^x)}\geq c_0, 
   % \end{equation*}
   % where 
   % with $e_1=(1,0,0)\in\R^3$.
   % Using the homotopy invariance of the Brouwer degree and again \eqref{eq:33}, we also get
   \begin{equation}
      \liminf_{h\to 0} \inf_{x\in \G^1}\frac{\leb^3\left(
            \left\{z\in B_h^x:\deg(y_h|_{B_h^{x'}},\partial B_h^{x'},y_h(z))=1\right\}
         \right)}{\leb^3(B_h^x)}\geq c_0\,. \label{eq:34}
   \end{equation}
   Note that $\deg(y_h|_{B_h^{x'}},\partial B_h^{x'},y_h(z))=1$ is sufficient to
   conclude that $y_h$ is not 1-to-1 at $z\in B_h^x$. Hence, \eqref{eq:34} proves
   \eqref{eq:12} for $x\in \G^1$.\\
   {\bf Step 4.} Proof of \eqref{eq:12} for $x\in \G^2$. This closely parallels the
   previous step, this time using the balls $B_h^x=B(x,h/(20C^*))$, $B_h^{\bar x}=B(\bar x, h/(20 C^*))$. As in the last step, we use Lemma~\ref{lemma:morrey} to obtain rigid motions $A_h^x, A_h^{\bar x}$ that satisfy
   \begin{equation*}
      \sup_{z\in B_h^x}|y_h(z)-A_h^x(z)|\leq C h^{1+\bar\e}\,,\quad
      \sup_{z\in B_h^{x'}}|y_h(z)-A^{\bar x}_h(z)|\leq C h^{1+\bar\e}\,.
   \end{equation*}
   Here, we even have
   \begin{equation*}
      \lim_{h\to 0}\inf_{x\in \G^2}\frac{1}{h}|A_h^x(x)- A^{\bar x}_h(\bar x)|=0,
   \end{equation*}
   and hence
   \begin{equation*}
      \lim_{h\to 0} \inf_{x\in \G^2}\frac{\leb^3\left(
            \left\{z\in  B_h^x:\deg(A^{\bar x}_h ,\partial B_h^{\bar x},A_h^x(z))=1\right\}
         \right)}{\leb^3(B_h^x)}= 1, 
   \end{equation*}
   and
   \begin{equation*}
      \lim_{h\to 0} \inf_{x\in \G^2}\frac{\leb^3\left(
            \left\{z\in  B_h^x:\deg(A^{\bar x}_h ,\partial B_h^{\bar x},A_h^x(z))=1\right\}
         \right)}{\leb^3(B_h^x)}= 1.
   \end{equation*}
   This proves~\eqref{eq:12} for $x\in \G^2$ and completes the proof of the proposition.
\end{proof}

\begin{proof}[Proof of Theorem~\ref{mainthm}.]
Let $z_h\in W^{1,2}(\Omega_h,\R^3)$ be defined by 
\[
   z_h(x',hx_3)=y_h(x',x_3)\quad\text{ for all }x'\in S, x_3\in[-1/2,1/2].
\]
By \eqref{eq:8}, 
   \begin{equation}
      \label{zhSO3}
      \|\dist(\nabla z_h,{\mathrm{SO}(3)})\|^2_{L^2(\Omega_h)}\leq C h^{2+\e}\,.
   \end{equation}
{\bf Step 1.} Approximation  by Lipschitz functions. 
Using~\eqref{zhSO3}, 
\begin{align*}
   \int_{\{|\nabla z_h|>2\sqrt{3} \}}|\nabla z_h|\d x \leq & \frac{1}{2\sqrt{3}}\int_{\{|\nabla z_h|>2\sqrt{3}\}}|\nabla z_h|^2\d x\\
   \leq & \frac{4}{2\sqrt{3}}\int_{\{|\nabla z_h|>2\sqrt{3}\}}\dist^2(\nabla z_h,{\mathrm{SO}(3)})\d x\\
   \leq & C h^{2+\e}
\end{align*}
We apply Theorem~\ref{thm:zhang} (with $K\to 2\sqrt{3}$, $f\to z_h$, $\e\to C h^{2+\e}$) and obtain $\tilde z_h\in W^{1,\infty}(\Omega_h,\R^3)$ such that
\begin{align}
\label{eq:2}
\left|\{z_h\neq\tilde z_h\}\right|\leq  & Ch^{2+\e}\\
\|\nabla \tilde z_h \|_{L^\infty(\Omega_h)}\leq & C
\end{align}
% Further, let $\tilde u_h$ be the Lipschitz function one obtains by applying
% Theorem~\ref{aikawa_thm2} with $f=u$, $K=h^{-1}$, and 
% \[
% \e=\int_{|\nabla u|\geq h^{-1}}|\nabla u|\d x\,.
% \]
% Note that $\tilde u_h\to u$ in $W^{1,2}(S,\R^3)$ for $h\to 0$. 
% Now set
% \begin{align}
% v_h= & \tilde z_h(\cdot,0)\label{newvhdef}\\
% u_h= & M_h\circ v_h\label{newuhdef}\\
% %w_h= & (M_h)^{-1}\circ u\label{whdef}
% \end{align}

%% \begin{equation*}
%% 	\begin{split}
%% 		\tilde{y}_{h}(x'):= \bar y_{h}(x',th)
%% 	\end{split}
%% \end{equation*}
{\bf Step 2.} Extension to a sphere. By Definition~\ref{def:simple_interpenetration}, % there exist disjoint simply connected Lipschitz sets $U_1,U_2$ such that $u_1:=u|_{U_1}$ and $u_2:=u|_{U_2}$ interpenetrate and
there exists an extension $\hat u_1:\hat U_1\to \R^3$ such that eq.~\eqref{eq:32} is fulfilled. For $\delta>0$, let 
\[
   U_{1,\delta}:=\left\{ x\in U_1:\dist(x, \partial U_1)<\delta\right\}\,.
\]
Now we choose $\delta$ so that 
\[
   \overline{u_1(U_{1,\delta})}\cap \overline {u_2(U_2)}=\emptyset\,.
\]
Such a choice of $\delta$ is possible by the fact that $u_2$ interpenetrates $u_1$, cf.~Definition~\ref{def:simple_interpenetration}. Set $\bar\delta:=\dist(u_1(U_{1,\delta}),u_2(U_2))$.
Next let $\chi_\delta\in C^{\infty}_0(U_1)$ with $\chi_\delta=1$ on $U_1\setminus U_{1,\delta}$ and $\|\nabla\chi_{\delta}\|_{L^\infty}<C\delta^{-1}$. Set
\begin{equation}
   \label{uhathdef}
   \hat u_{1,h}(x)=\begin{cases}\tilde z_h(x,0) & \text{if } x\in U_1\setminus U_{1,\delta}\\
      \chi_\delta(x) \left(\tilde z_h(x,0)\right)+(1-\chi_\delta(x))u(x) & \text{if } x\in  U_{1,\delta}\\
      \hat u_1(x)& \text{if } x\in \partial \hat U_1\setminus U_1\end{cases}
\end{equation}
and
\begin{equation}
   \label{eq:3}
   u_{2,h} \,= \tilde z_h(\cdot,0)|_{U_2}\,.
\end{equation}
%% and 
%% \begin{equation*}
%% \begin{array}{rlcrl}
%% \hat v_{1,h}= & (M_h)^{-1}\hat u_{1,h}&\quad& v_{2,h}= & (M_h)^{-1} u_{2,h}=\tilde z_h(\cdot,0)|_{U_2}\\
%% \hat w_{1,h} = & (M_h)^{-1}\hat u_{1}&\quad& w_{2,h}= & (M_h)^{-1} u_2
%% \end{array}
%% \end{equation*}

{\bf  Step 3.} Convergence of Brouwer degree in $L^1$.\\
Let
\[
% E_{2,\e}^*:=\left(\{x\in U_2:\dist(x,\hat u_1(\hat U_1))\leq \e\}\right)\,.
E:=\{x\in U_2:\,u_2(x)\in\hat u_1( \hat U_1)\}\,.
\]
% Let $\dlim:  U_2\to \R$ be defined by
We claim that 
\begin{equation}
\label{eq:45}
	\begin{split}
		\deg(\hat u_{1,h},&\hat U_1,u_{2,h}(\cdot))\to
                \deg(\hat u_1,\hat U_{1},u_2(\cdot)) \\
		&\text{ in } L^1(U_2\setminus E)\quad \text{as $h\to 0$}\,.
	\end{split}
\end{equation}
We prove this claim by  a homotopy argument.\\
By definition of $z_h$ and \eqref{eq:7},
\[
\fint_{[-h/2,h/2]}\d x_3 z_h(\cdot,x_3)\wto u_i \text{ in } W^{1,2}(U_i,\R^3)\,.
\]
By definition of $\tilde z_h$, this holds also true if $z_h$ is replaced by
$\tilde z_h$. By the uniform Lipschitz bound \eqref{eq:2} on $\tilde z_h$, we
also have
\[
\tilde z_h(\cdot,0)\wto u_i \text{ in } W^{1,2}(U_i,\R^3)\,.
\]
By the definitions of $\hat u_{1,h},u_{2,h}$ in \eqref{uhathdef} and
\eqref{eq:3}, we get
\begin{equation}
   \label{eq:6}
   \hat u_{1,h}\wto \hat u_1 \text{ in }W^{1,2}(\hat U_1,\R^3)\quad \text{and}\quad
   u_{2,h}\wto u_2 \text{ in }W^{1,2}(U_2,\R^3)\,.
\end{equation}
% For $\hat
% u_{1,h}$ and $u_{2,h}$, we have the following Lipschitz bounds:
% \[
% |\nabla \hat u_{1,h}|\leq \max |\nabla \tilde z_h|,|\nabla \hat u_1|, C\delta^{-1}|\tilde z_h-u|

Since the uniform Lipschitz bound holds for $\tilde z_h$, there also exist
uniform Lipschitz bounds for $\hat
u_{1,h}$ and $u_{2,h}$ by definition of the latter two. Hence the weak
convergence in \eqref{eq:6} is also true in
$W^{1,p}$ for every $1<p<\infty$. By the compact Sobolev embedding, we have
$\hat u_{1,h}\to \hat u_1$ and $u_{2,h}\to u_2$ in $C^{0,\alpha}$ for every
$0<\alpha<1$, and in particular, we have uniform convergence.\\
Set $E_\e:=\{x\in U_2\setminus E:\dist(x,U_2)$.
Since $E$ is relatively closed in $U_2$, we have $\leb^2(E_\e)\to 0$ as $\e\to 0$. The claim \eqref{eq:45} follows from the continuity of the degree function in
the first and the third argument with respect to uniform convergence. 

{\bf Step 4.} Application of isocapacitary inequality and passage back to 3d.
By the definition of  interpenetration 
(Definition~\ref{def:simple_interpenetration}), there exist $k_1,k_2\in\N$,
$k_1\neq k_2$ and some $C>0$ such that
\[
   \left|\{x\in U_2:\deg(\hat u_1,\hat U_1,u_2(x))=k_i\}\right|>C\,\text{ for }i=1,2\,.
\]
Hence by step 3, there exists $h_0>0$ such that
\begin{equation}
\label{eq:11}
\left|\{x\in U_2:\deg(\hat u_{1,h},\hat U_1,u_{2,h}(x))=k_i\}\right|>C\,\text{for }i=1,2\,
\end{equation}
for $h<h_0$ (which we assume from now on). Let 
\[
A_h:=\{x\in U_2:\deg(\hat u_{1,h},\hat U_1,u_{2,h}(x))=k_1\}
\]
and let $U^\circ_2$ denote the interior of $U_2$. Then by \eqref{eq:11},
$\min(|A_h\cap U^\circ_2|,|U^\circ_2\setminus A_h|)>C$. We apply Lemma~\ref{lemma:isoperimetric-capacity} and obtain
\begin{equation} %---{{{
   \label{eq:17b}
   \begin{split}
\kap_1(\partial A_h\cap U^\circ_2)>C\,.
   \end{split}
\end{equation} %---}}}
On the other hand, $x\in \partial A_h\cap U^\circ_2$ implies 
\[
   u_{2,h}(x)\in \partial\{y\in\R^3:\deg(\hat u_{1,h},\hat U_1,y)=k_1\}
\]
and hence by \eqref{eq:9}, 
\[
\begin{split}
   \partial A_h\cap U^\circ_2\subset & \{x\in U_2:\,u_{2,h}(x)\in \hat
   u_{1,h}(\hat U_1)\}\,.
\end{split}
\]
By the definition of $\hat u_{1,h}$ in \eqref{uhathdef} and the uniform convergence $\hat
u_{1,h}\to\hat u_1$, $u_{2,h}\to u_2$, we may assume that $\dist(x,\partial
U_2)>\delta$ whenever $u_{2,h}(x)\in \hat u_{1,h}( \hat U_1)$, whence
$\hat u_{1,h}(x)= u_h(x)$ for $x\in \partial A_h\cap U_2^\circ$ and
\[
\partial A_h\cap U^\circ_2
\subset  F_h:=\{x\in U_2: \text{ there exists }\bar
x\text{s.t.\ }u_h(x)=u_h(\bar x) \text{ and }|x-\bar x|>2h\}\,.
\]
By~\eqref{eq:17b},  we have proved
\begin{equation*} %---{{{
   \begin{split}
      \kap_{1}\big( F_h \big) >C\,.
   \end{split}
\end{equation*} %---}}}
By this last inequality and the results from step 1, the conditions of Proposition~\ref{prop:05291369832965} are fulfilled and we can apply it to $\tilde{z}_{h}$ and obtain that
\begin{equation*} %---{{{
    \begin{split}
       \leb^{3}\left( \insieme{x:\ \tilde{z}_{h} \text{ is not $1$-to-$1$} }\right) > c h^{2}.
    \end{split}
 \end{equation*} %---}}}
 
 By \eqref{eq:2} and the definition of $z_h$, one has that 
 \begin{equation*} %---{{{
    \begin{split}
        \leb^{3}\left( \insieme{x:\ z_{h} \text{ is not $1$-to-$1$} }\right)= \leb^{3}\left( \insieme{x:\ y_{h} \text{ is not $1$-to-$1$} }\right) > ch^{2},
    \end{split}
 \end{equation*} %---}}}
 which concludes the proof. 
\end{proof}

\section{Application to plate theories derived as $\Gamma$-limits}
\label{sec:appl} 
As before, let $S\subset \R^2$ be open and bounded and $\Omega=S\times[-1/2,1/2]$. 
We  define the elastic energy of a 3-dimensional body. Let the inhomogeneous stored energy $W:\Omega\times
\R^{3\times 3}\rightarrow [0,\infty)$ satisfy
\begin{itemize}
\item[(i)]$W(x,FR)=W(x,F)$ for all $R\in SO(n)$
\item[(ii)]$W(x,\id_{3\times 3})=0$
\item[(iii)]$W(x,F)\geq c\, \text{dist}^2(F,SO(3))$ for some uniform
  constant $c$
\item[(iv)]$ W\in C^2(S,\mathcal T)$ where $\mathcal{T}$ is an $\epsilon$-neighbourhood of $SO(3)$.
\item[(v)]$W(x,F)=W(z,F)$ if $(x-z)\| e_3$.
\end{itemize}
We introduce the quadratic forms $Q_3:\Omega\times\R^{3\times 3}\rightarrow\R,\,
Q_2:S\times\R^{2\times 2}\rightarrow\R$ by
\begin{align*}
Q_3(\bar x,\bar F)&=\frac{D^2W(x,F)}{DF^2}|_{x=\bar x,F=\id}(\bar
F,\bar F)\\
Q_2(\bar x',\bar F')&=\min\left\{Q_3(\bar x,F'+a\otimes e_3+e_3\otimes a):a\in \R^3\right\}
\end{align*}
The integral of $W$ satisfying properties (i) through (v) above is the
(rescaled) elastic energy functional 
\begin{equation*}
%\label{elenIh}
\begin{array}{rrl}
I_{h}:&W^{1,2}(\Omega,\R^3)\to&\R\\
&y\mapsto&\int_{\Omega} W(x,\nabla_h y(x)) \d x\,.
\end{array}
\end{equation*}
The penalization of interpenetration of matter is expressed in a modification of  the 3d energy functional $I_h$, assigning infinite energy to
non-physical deformations. We define $\bar
I_{h}:W^{1,2}(\Omega,\R^3)\to\R\cup\{+\infty\}$ by
% We  define a modification of the energy functional \eqref{elenIh} that
%  penalizes interpenetration of matter.
\begin{equation}
\label{elenIhbar}
\bar I_h(y)=\Bigg\{\begin{array}{ll}\int_{\Omega} W(x,\nabla_h y(x)) \d x&\text{ if }y
  \text{ is invertible a.~e.}\\+\infty&\text{ else.}\end{array}
\end{equation}
\subsection{Contractive maps}
\label{sec:contract}
In \cite{MR2358334}, the $\Gamma$-limit of the functional $h^{-\beta}I_{h}$ for the scaling regime $0<\beta<5/3$
has been derived (using results from \cite{MR1365259}). The result can be stated
as follows: We say $y_h\in W^{1,2}(\Omega_h,\R^3)$ converges uniformly to
$u\in W^{1,2}(S,\R^3)$ as $h\to 0$ if
\[
\lim_{h\to 0} \text{ess sup}_{(x_1,x_2,x_3)\in\Omega_h}|y_h(x_1,x_2,x_3)-u(x_1,x_2)|=0\,.
\]
Further, we say that $u\in W^{1,\infty}(S,\R^3)$ is \emph{short} if
\[
\nabla u^T\nabla u\leq \id_{2\times 2} \quad\text{a.~e.}
\]
i.e., $\id_{2\times 2}-\nabla u^T\nabla u$ is positive semi-definite almost everywhere.
The $\Gamma$-convergence result from \cite{MR2358334} can be stated as saying
that for $0<\beta<5/3$,
\[
\left(\Gamma-\lim_{h\to 0} h^{-\beta}I_h\right)(u)=\begin{cases}0 & \text{ if }
  u \text{ is short}\\ +\infty & \text{ else}\,,\end{cases}
\]
where the $\Gamma$-limit is taken with respect to uniform convergence. In fact,
it could just as well have been formulated for weak convergence in $W^{1,2}(\Omega,\R^3)$
(see the discussion in \cite{MR2358334}). This result includes the trivial lower
bound 
\[
\liminf_{h\to 0} h^{-\beta}I_h(y_h)\geq 0
\] 
for sequences $y_h$ that converge
towards a short map $u$.
The application of Theorem~\ref{mainthm} immediately yields the following corollary, that is a sharper
lower bound for $h^{-\beta}\bar I_h$ for $1<\beta<5/3$.
\begin{corollary}[to Theorem~\ref{mainthm}]
Let $1<\beta$, $u\in W^{1,\infty}(S,\R^3)$, and let $U_1,U_2\subset S$ be disjoint simply connected Lipschitz domains
such that with $u_1:=u|_{U_1}$, $u_2:=u|_{U_2}$,  $u_2$ interpenetrates $u_1$. Further let $y_h\in W^{1,2}(\Omega_h)$ converge uniformly to $u$.
Then
\[
\liminf_{h\to 0} h^{-\beta} \bar I_h(y_h)=+\infty\,.
\]
\end{corollary}

\subsection{Nonlinear bending theory}
\label{sec:nlb}
In \cite{MR1916989}, the nonlinear Kirchhoff plate theory was obtained as the
$\Gamma$-limit of the scaled functional $h^{-2}I_h$. Nonlinear plate theory can
be defined as follows:

Let the set of $W^{2,2}$-isometries of $S$ into $\R^3$ be denoted by
\[
\A=\{u\in W^{2,2}(S,\R^3):\nabla u^T\nabla u=\id_{2\times 2}\}\,.
\]
Further, the second fundamental form is given by
\[
{\mathrm{II}}_{[u]}=\nabla u^T\cdot \nabla \nu\,,
\]
where $\nu=u_{,1}\wedge u_{,2}$ is the normal of the isometry $u$.
% Let $\iota:S\to\Omega$, $\iota(x')=(x',0)$.
Nonlinear plate theory may be defined via the energy  functional
\[
\begin{array}{rrl}\Ik:&W^{2,2}(S,\R^3)\to&\R\cup\{+\infty\}\\
&u\mapsto&\begin{cases} \frac{1}{24}\int_S Q_2(x',\sff_{[u]}) \d x' &\text{ if } u
   \in \A\\
+\infty&\text{ else}\end{cases}
\end{array}
\]
The limiting deformations with finite bending energy will be the set of $y\in
W^{1,2}(\Omega,\R^3)$ such that there exists $u\in\A$ with
\begin{equation}
y(x',x_3)=u(x')\text{ for a.~e.\ }x'\in S, x_3\in[-1/2,1/2]\,.\label{eq:27}
\end{equation}
We define the auxiliary functional $\Iko:W^{1,2}(\Omega,\R^3)\to\R\cup\{+\infty\}$ by
\begin{equation}
\label{Ki3d}
\Iko(y)=\begin{cases} \Ik(u) &\text{ if }\exists u\in \A \text{ such that
    eq.~\eqref{eq:27} holds}\\
\infty&% \text{ if }y\not\in \mathcal A
\text{ else. }
\end{cases}
\end{equation}
% where $\iota:S\to \Omega$, $\iota(x')=(x',0)$.
\begin{theorem}[\cite{MR1916989}, $\Gamma-\liminf$-inequality]
\label{thm:nlbending}
% Let the stored energy function $W$ satisfy conditions (i)-(v),
% \emph{(Compactness.)} Let $y_h$ be a sequence in $ W^{1,2}(\Omega,\R^3)$ % be a sequence with $y_h\in W^{1,2}(\Omega_h,\R^3)$ 
% with
% \[
% \limsup_{h\to 0} h^{-2}\int_{\Omega} \dist^2(\nabla_h y_h,SO(3))\d x<\infty\,.
% \]
% Then there exists $u\in \A$ and a subsequence (not relabeled) such that
% \[
% y_h\to y \text{ in }W^{1,2}(\Omega,\R^3)
% \]
% where 
% \begin{equation}
% y(x',x_3)=u(x')\text{ for a.~e. }x'\in S, x_3\in[-1/2,1/2]\,.\label{eq:26}
% \end{equation}
% % $P_h:\Omega\to\Omega_h$, $P_h(x',x_3)=(x',hx_3)$, $P_S:\Omega\to S$, $P(x',x_3)=x'$.\\
%\emph{(Lower bound.)} 
Let $y_h,y\in W^{1,2}(\Omega,\R^3)$,  $y_h\wto y$ in
$W^{1,2}(\Omega,\R^3)$. Then  % and eq.~\eqref{eq:26} holds for some $u\in \A$, 
\[
\liminf_{h\to 0} h^{-2}I_h(y_h)\geq \Iko(y)\,.
\]
% \emph{(Upper bound.)} For every $y\in W^{1,2}(\Omega,\R^3)$, there exists a sequence
% $\{y_h\}_h\subset W^{1,2}(\Omega,\R^3)$ such that $y_h\to y$
% in $W^{1,2}(\Omega,\R^3)$ and
% \[
% \lim_{h\to 0}h^{-2}I_h(y_h)=\Iko(y)\,.
% \]
\end{theorem}
The application of Theorem~\ref{mainthm} to nonlinear bending theory yields the
following sharper version of the lower bound for $h^{-2}\bar I_h$.

\begin{corollary}[to Theorem~\ref{mainthm}]
\label{thm:nlbendinginter}
Let $u\in \A$, and let $U_1,U_2\subset S$ be disjoint simply connected Lipschitz domains
such that with $u_1:=u|_{U_1}$, $u_2:=u|_{U_2}$,  $u_2$ interpenetrates $u_1$.
Further, let  $y_h\wto y$ in $W^{1,2}(\Omega,\R^3)$, with
\begin{equation*}
\limsup_{h\to 0}h^{-2}\|\dist(\nabla_h y_h,{\mathrm{SO}}(3))\|^2_{L^2(\Omega)}<\infty
%\label{eq:36}
\end{equation*}
and
\[
y(x',x_3)=u(x') \quad\text{ for a.~e.\ }\quad x'\in S,\, x_3\in[-1/2,1/2]\,.
\]
Then
\[
\liminf_{h\to 0} h^{-2} \bar I_h(y_h)=+\infty\,.
\]
\end{corollary}

\begin{remark}
In the case $\beta>2$, a consequence of the compactness part of the $\Gamma$-convergence result for $h^{-\beta} I_h$ in \cite{MR2210909} is the following: Whenever  
\[
\limsup_{h\to
  0}h^{-\beta}I^h(y_h)<\infty\quad\text{and}\quad y_h\wto y \text{ in
}W^{1,2}(\Omega,\R^3)\,
\]
 then
$y$ is (up to a rigid motion) just the
projection onto the first two components, $y(x)=x'$. 
This indicates that if $S$ is connected, and $U_1,U_2\subset S$ are disjoint
subsets, it is impossible to state sufficient conditions for the limits
that assure that $y_h|_{(U_1\cup U_2)\times[-1/2,1/2] }$ is 2-to-1 on a
set of positive measure. One can still create a setting in which our main result is applicable,  considering reference sets $S$ with more than one connected
component. We refrain from doing so here for the sake of brevity.
\end{remark}

\bibliographystyle{plain}
\bibliography{inter}
\end{document}